\documentclass{amsart} 

\usepackage[centertags]{amsmath} 
\usepackage{amsfonts,amssymb,amsthm,amscd,latexsym,eufrak}
\usepackage[mathscr]{euscript}
\usepackage[all]{xy}
\usepackage{hyperref} 

\newtheorem{theorem}{Theorem}[section]
\newtheorem{corollary}[theorem]{Corollary}
\newtheorem{proposition}[theorem]{Proposition}
\newtheorem{lemma}[theorem]{Lemma}
\theoremstyle{definition}
\newtheorem{definition}[theorem]{Definition}
\newtheorem{example}[theorem]{Example}
\newtheorem{remark}[theorem]{Remark}

\numberwithin{equation}{section}        

\newcounter{zahl}%
    {\end{list}}%

\newcommand{\co}{\colon\thinspace}
\newcommand{\mc}[1]{\ensuremath{\mathcal{#1}}}
\newcommand{\toh}[1]{\ensuremath{\stackrel{#1}{\rightarrow}}}

\newcommand{\wt}[1]{\widetilde{#1}}
\newcommand{\ho}[1]{\ensuremath{{\rm Ho}({#1})}}
\newcommand{\bu}{\bullet}%
\newcommand{\Hom}{\ensuremath{{\rm Hom}}}%
\begin{document}
\begin{abstract}
	We show that, for a principal ideal domain $R$, the homotopy category of Bousfield $R$-local spaces injects fully faithful into a homotopy category of simplicial pointed flat coalgebras.  
\end{abstract}
\author{Manfred Stelzer}
\footnote{Manfred Stelzer, Institut f\"ur Mathematik, Universit\"at Osnabr\"uck, Albrechtstrasse 28a, D-49076, Osnabr\"uck, Germany.
e-mail:mstelzer@uni-osnabrueck.de }
\title{Pointed integral coalgebras and  $\bold{R}$-local  homotopy theory}
\maketitle
\section{\large\textbf{Introduction}}
\subsection{}

 Categories of algebras and coalgebras have been used to model classes of topological spaces up to homotopy by many authors with great success.

  In this paper we investigate the situation for simplicial  coalgebras over a principal ideal domain.  \\
 Our  main theorem reads:\\
 \begin{theorem}\label{mainmain}\begin{it} Let $R$ be a localization $\mathbb{Z}[S^{-1}] $ of the integers at a set of primes $S$  or a prime  field. The simplicial chains over $R$ define a full and faithfull functor  from the homotopy category of   Bousfield $R$-local spaces to a homotopy category of a model category of flat  pointed simplicial cocommutative $R$-coalgebras.\end{it}
\end{theorem} 

The notion of a pointed coalgebra is explained below.  
 \\

For a separably closed field $K$ such a result was obtained by Goerss in \cite{G}.
 At the core of his argument is the fact that, over a separable field $K$,  the sum of the simple subcoalgebras  of a given  coalgebra $C$, splits off naturally.
 This sum of simple subcoalgebras  identifies with the \'{e}tale  part of $C$.
   In case $K$ is algebraically closed, the \'{e}tale  part is a sum with factors the  the ground field with trivial coalgebra structure.
Due to ramification, there are additional simple coalgebras over
  a general ring $R$.

  To deal with this fact, we restrict attention to the class of pointed coalgebras. By definition, a pointed coalgebra over $R$ admits only subcoalgebras isomorphic to $R$ as simple subcoalgebras.  In the  literature on coalgebras over a field, this class is well studied at least since Sweedler's classic book. It goes without saying that  this notion of pointedness differs from  the  one in use in the homotopy theory literature.\\
  Since we want the homotopy theory of simplicial coalgebras to behave well with respect to the tensor product, we consider only flat coalgebras. There is a monoidal model structure on simplicial flat coalgeras \cite{St}. \\
     In this setting  the  subobjects with flat quotient are most relevant. These are the pure subcoalgebras. The pure monomorphisms make up the cofibrations in the model  structure mentioned above.. So the notion of simple coalgebra and pointed coalgebra has to be modified to  that of a pure simple and (pure)  pointed coalgebra. For this class, we are able to generalize a large part of the elementary classical coalgebra structure theory over fields. This leads to the  following result which informs our main theorem.
  
  \begin{theorem}\begin{it}\label{natsplit} Let $R$ be a principal ideal domain and $C$  a flat  pointed cocommutative and coassociative coalgebra over $R$. Then   the sum of the pure simple subcoalgebras of $C$ splits of naturally.
  \end{it}	\end{theorem}

Let us comment on the notion of  pointed coalgebra.
For general coalgebras over a field the splitting of \ref{natsplit}  can be made to work by passing to an algebraic closure. In this case the simple objects in sets and in coalgebras match up.
Since the homotopy category of Bousfield $R$-local spaces does only depend on the core of $R$, this is good enough. Over a more general ring $R$ this strategy does not work as examples in section 6. show. There are additional simple coalgebras,  
  
   The passage to pointed objects is a way to avoid these extra "points".  \\
  

\subsection{ }
	\textbf{Historical references}\\
	
 After the seminal work of Quillen and Sullivan \cite{Q2},\cite{Su}, in which differential graded coalgebras and algebras were employed as  models of rational homotopy theory, there has been large progress for the algebraic description of the  p-complete homotopy  theory. Simplicial coalgebras and cosimplicial algebras over an algebraically closed field were used to model p-complete spaces by Goerss and Kriz \cite{G},\cite{Kr} respectively, A generalization of Goerss theorem to coalgebras in simplicial presheaves was given by Raptis in \cite{R}. A little later, Mandell proved that the homotopy category of finite type nilpotent p-complete spaces is modeled by a category of $E_{\infty}$-algebras over $\bar{\mathbb{F}}_p$ \cite{Ma2}.
  Very recently, Bachmann and Burglund \cite{BaBu} got rid of the finite type assumption in \cite{Ma2} using $E_{\infty} $-coalgebras. Mandell went on to show that 1-connected integral homotopy types of finite type are determined by   integral $E_{\infty}$-algebras \cite{Ma}. There he asked for a model which describes the maps in the homotopy category as well \cite{Ma}, \cite{Ma3}. Also  George Raptis asked for for such an integral extension of Goerss main result in \cite{R}. More recently, Yuan constructed a fully faithful functor from the homotopy category of finite and  1-connected spaces to a category of algebras over the sphere spectrum \cite{Y}. Doing away with all assumptions on the fundamental group,  Rivera, Wiestra and Zeinalian showed that integral homotopy types can be distinguished by integral  simplicial coalgebras \cite{R.W.Z.}.
 Again, the question on the maps remained open.
 Then Raptis and Rrivera  constructed in \cite{R.R.} model categories for simplicial coalgebras with equivalences which detect the fundamental group in general. Over algebraiclly closedfields a generalization of Goerss main resuct is proved.  \\
 Finally, Horel gave a model for the integral homotopy category of nilpotent spaces of finite type in \cite{Horel}.
 It is shown in \cite{Horel} that cosimplicial  binomial algebras serve as a faithful model. Our main result seems to be in a similar relation to the one of Horel as  Goerss paper relates to the one of Kriz.
  
A version for $\infty$-objects is not known in general. 
 One source  of the problem is that  the  multiplicative structure on the integral cohomology of Eilenberg-MacLane spaces is not well understood \cite{Per}. The same can be said about  the integral cofree  $E_{\infty}$-coalgebras.
\\

This paper is organized as follows. In section 2. we collect some results on purity from the literature and provide some results needed in the sequel. In sections 3. to 8. we generalize structure results on coalgebras over fields to pure pointed coalgebras over a principal ideal domain.
 The important notion of pure irreducibility is introduced in section 3..
Group-like elements in flat coalgebras are the subject of section 4. It is shown that they span a pure subcoalgebra.
In section 5. we show that a pointed coalgebra decomposes in pure irreducible components. The splitting asserted in \ref{natsplit} is proved in section 6.  Filtrations  of flat coalgebras are investigated in section 7. The important pure coradical filtration is the subject of section 8. Categorical properties of pointed flat coalgebras are studied in section 9. These results are applied in
section 10. which is devoted to the proof that simplicial pointed flat coalgebras carry a convenient  model structure.
Here we rely heavily on \cite{St} where such a model structure was constructed for unpointed objects.  There in turn we used the flat model structure on flat chain complexes due to Gillespie \cite{Gil 1} as a stepping stone.\\ 
 We assume that the reader is familiar with
the theory of model categories. For background material, we
recommend the  monographs of Hirschhorn \cite{Hi} and Hovey \cite{Hovey2}.

 Finally, the proof of  \ref{mainmain} is established in section 11. Given the structural results on flat pointed coalgebras, the argument is very much the same as the one in \cite{G}. In section 12. we relate pointed flat coalgebra to binomial rings. Thereby making contact to the model of Horel.  \\

\textbf{Conventions:}\\
All rings are commutative with unit. A ring $R$ is fixed throughout the paper.  If nothing else is mentioned, it is  supposed to be principal ideal domains.
All unadorned tensor products are over the ground ring $R$. The quotient field of the domain $R$ will be denoted by $K$.
\section{\large\textbf{Recollections on Purity and Purification}}
In this section we summarize some results about purity from the literature  \cite{P}, \cite{F.S.1},\cite {B.R.}, \cite{St}.
Denote the category $R$-modules and of flat $R$-modules by  $\mathbf{Mod}^{}$ and $\mathbf{Mod}^{flat}$ respectively.  
 Both categories are well known to be   complete and  cocomplete.\\

\begin{definition}\begin{it} A submodule $N$ of an $R$-module $M$ is called pure if 
		\[0\to N\otimes A\to M\otimes A\] is exact for every $R$-module $A$.
		An $R$-module $I$ is called pure injective if for any pure pair \[0\to N\to M\] the sequence \[Hom(M,I)\to Hom(N,I)\to 0\] is exact.\end{it}\end{definition} 
	
	
	\begin{definition}\begin{it} Let  $M\in \mathbf{Mod}^{flat}$ and $N$ a submodule of $M$.
			The purification $\widetilde{N}$  of  $N$ in $M$  is the intersection of all pure submodules of $M$ which contain $N$.\end{it}
	\end{definition}
	
	For $R$ a principal ideal domain  (or more generally a Pr\"ufer  
	domain) purity can be described by a relative divisibility property. Recall that a submodule $U$ of $M$ is called a RD-submodule if the equality $$rU=rM\cap U$$ holds for all $r\in R$. Then the classes of pure and RD submodules coincide over $R$. Hence, $\widetilde{N}$  is the relative divisibility hull of $N$.\\ 

 We record some properties of pure monomorphism from \cite {B.R.}.
They hold over a general ring $R$.

\begin{itemize}\item a composition of two pure monomorphisms is a pure monomorphism,
	\item if $fg$ is a pure monomorphism then $g$ is a pure monomorphism,
	\item pure monomorphisms are stable under pushouts,
	\item pure monomorphisms are closed under transfinite composition, 
	\item pure monomorphism are precisely directed colimits of split monomorphisms in the arrow category of $R$-modules.\\ 
\end{itemize}
See \cite[2.10.]{St} for:
\begin{lemma}\label{intersections}\begin{it}
 Let  $M\in \mathbf{Mod}^{flat}$ and $A,B$   pure $R$-submodules of $M$.
Then $A\cap B$ is pure in $M$.
\end{it}\end{lemma}

How tensor product and intersection of pure modules relate is the subject of \cite[lemma 2.15]{St}: 

\begin{lemma}\begin{it} 
		 Let  $M\in \mathbf{Mod}^{flat}$ and $A,B$   pure $R$-submodules of $M$.
				Then the inclusion \[i:(A\cap B)\otimes (A\cap B)\to(A\otimes A)\cap (B\otimes B)\] is an isomorphism. \end{it} 
	 \end{lemma}
 
 The proof uses only that the intersection of pure modules is pure and that the assertion holds over a field.
 This is true for arbitrary intersections.  The first fact is straight forward the second can be found in \cite[Exercise 1.2.8. ]{Rad}.  Hence, there is the following generalization:
 
  \begin{lemma}
 	Let $\{A_i\}_{i\in I}$, $\{B_i\}_{i\in I}$ families of pure submodules of  $M\in \mathbf{Mod}^{flat}$. There is an isomorphism
 	$$(A_i \cap_{i\in I} B_i)\otimes (A_i \cap_{i\in I} B_i) \cong \cap_{i\in I} (A_i \otimes B_i)$$
 \end{lemma}
 
 
 \begin{lemma}\begin{it}\label{infint} 
 		Let  $M\in \mathbf{Mod}^{flat}$ and $A_i ,B_i$  families   pure $R$-submodules of $M$.
 		Then the inclusion \[i:(A\cap B)\otimes (A\cap B)\to(A\otimes A)\cap (B\otimes B)\] is an isomorphism. \end{it} 
 \end{lemma}

In this respect, purification behaves as follows \cite[lemma 2.17.]{St}:

\begin{lemma}\label{i iso}\begin{it}Let  $M\in \mathbf{Mod}^{flat}$ and $N$    submodule of $M$.  There is an isomorphism \[j:\widetilde{(M\otimes N)\cap (N\otimes M)}\to (M \otimes\widetilde{N} ) \cap (\widetilde{N}\otimes M).\]

\end{it}\end{lemma}

The following is a slight generalization of \cite[ Theorem 2.16.]{St}. It will be needed later on.

\begin{lemma}\label{45}\begin{it} Let  $M\in \mathbf{Mod}^{flat}$ and $A,B,C,D$ pure submodules of $M$. Then the map 
		$$i:(A\cap B )\otimes (C\cap D) \to (A\otimes C)\cap (B\otimes D )   $$
		
		is an isomorphism.\end{it}
\end{lemma}	

\begin{proof}
	That the map $i$ is pure may be seen by the argument in \cite[ Lemma 2.11.]{St}. Hence, the cokernel of $i$ is flat. It is enough to see that $i\otimes K$ is an isomorphism.  The assertion holds for the case that $R=K$.  One reduces to this case as in \cite[ Theorem 2.16.]{St}.
\end{proof}

\begin{lemma}
	\label{pure mod p}\begin{it} Let $$i:N \to M$$ an injection of flat $R$-modules. Then $i$ is pure if and only if it induces an injection mod $p$ for all primes $p$ in $R$.\end{it}
\end{lemma}
\begin{proof} 	First, note that purity can be tested by tensoring with finitely presented modules and by considering all localizations at maximal ideals see \cite[4.89(2)]{La} and \cite[exercise 34 p.163]{La}.  Thus, we may assume that $R$ is a local.
	A  finitely presented module $M$
	decomposes uniquely as a finite direct sum of cyclic modules $R$ and  $R/p^r R $ with prime $p$.
	Consequently, we may test purity by tensoring with quotient rings  $R/p^r R$ with prime $p$.
	We claim that it is enough to consider the case $r=1$. To see this tensor the short exact sequence $$R/pR \rightarrowtail R/p^{k+1}R \twoheadrightarrow R/p^{k}R$$
	Tensoring with the flat objects $N$ and $M$ produces a diagram in which the rows are short exact sequences 
	
	\xymatrix{N\otimes R/pR \ar[d]\ar@{>->}[rr] &&
		N\otimes R/p^{k+1}R \ar@{->>}[rr]\ar[d] && N\otimes R/p^{k}R\ar[d] 
		\\
		M \otimes R/pR   \ar@{>->}[rr]^{} &&    M\otimes  R/p^{k+1}R
		\ar@{->>}[rr]^-{} &&   M\otimes  R/p^{k}R
	}
	
	and the assertion follows by induction and the sharp five lemma.
	\\
\end{proof}
\section{\large\textbf{Pure irreducibility} }
 
We denote the category of cocommutative coassociative $R$-coalgebras and the subcategory of coalgebras with flat underlying $R$-module by $\mathcal{CA}$ and
$\mathcal{CA}^{flat}$ respectively.

The classical fact that the intersection of two subcoalgebras of a given coalgebra over a field inherits the structure of a coalgebra, was generalized  \cite[lemma 3.9.]{St} for two pure subcoalgebras.


Using \ref{infint} this extends to:

\begin{lemma}\label{infcap} Let 
	$C\in \mathcal{CA}^{flat}$. The intersection $\cap_{i\in I} (D_i\cap F_i )$ of  pure subcoalgebras $D_i,F_i \subset C$ inherits the structure of a coalgebra. 
\end{lemma}

\begin{remark}Since the pushout of two pure subcoalgebras along their intersection is
	a pure subcoalgebra, we see that pure subcoalgebras have all the defining properties of closed subspaces in a topological space. Of course, one has to replace the empty set by the initial coalgebra. On the other hand, intersections and unions do not distribute.
\end{remark}

\begin{definition} \label{simple}
\begin{it} Let 
	   $C\in \mathcal{CA}^{flat}$. 
\begin{enumerate}
\item If $C$  admits no non-trivial pure subcoalgebras it is called pure simple.
\item $C$ is called  pointed, if  all pure simple subcoalgebras  are isomorphic to $R$.
\item A coalgebra    $C$ is called pure irreducible, if the intersection of all non trivial pure subcoalgebras 
 $D$,  $E$ is non trivial.
 \item A maximal pure irreducible subcoalgebra of  $C$ is called a pure irreducible component.
\end{enumerate} 
\end{it}\end{definition}

Write $\mathcal{CA}^{flat}_*$ for  the category of cocommutative, coassociative, flat and pointed $R$-coalgebras.\\
\begin{remark}
Since we have no use for the impure version, we simply speak of pointed coalgebras. 
\end{remark}

	For a field $K$ a finite pointed $K$-coalgebra  $C$ is isomorphic to the dual of  a finite local $K$-algebra 

Recall that  a flat $R$-module  $M$ is of finite rank $n$ if dim$M\otimes  K=n$. In \cite[theorem  3.6.]{St} the following generalization of the fundamental theorem for coalgebras over a field was obtained:

\begin{theorem}\label{colimits} A flat coalgebra $C$
	    is the filtered colimit of its pure subcoalgebras of finite rank.
\end{theorem}


.

\begin{lemma}\label{irreducible pure}\begin{it} Let  $C\in \mathcal{CA}^{flat}$. Then a pure irreducible component of $C$  is pure in $C$.\end{it}\end{lemma}
\begin{proof} Let $D\subset C$ be a pure irreducible component.
The purification $\tilde{D} $ is a subcoalgebra of $C$ by \cite[lemma 2.17.]{St}.
If two pure subcoalgebras of $ \tilde{D}$ have non-trivial intersection then the same holds in $ D$. Thus $\tilde{D}$  is an irreducible subcoalgebra which contains $D$.
Hence we find that $D=\tilde{D}$ \end{proof} 

\begin{lemma}\label{irreducible rules} Let $C\in \mathcal{CA}^{flat}$.
 \begin{it}\begin{enumerate} 
\item	 Any  pure simple subcoalgebra of $C\in \mathbf{C}^{flat}$ is of finite rank.\\
\item  Any non-zero pure subcoalgebra of $C$ contains a pure simple subcoalgebra.\\
\item  The coalgebra $C$ 	is pure irreducible if and only if it contains a unique pure simple subcoalgebra $D$.\end{enumerate} \end{it}\end{lemma}

\begin{proof} Let $D$ be pure simple.  Any element  $x\in D$ is contained in a pure subcoalgebra of finite rank.
This proves (1).\\
By (1) we may assume that $C$ is of finite rank. If $C$ is not pure simple there is a pure subcoalgebra $D$ of strictly smaller rank. Then (2) follows by induction on the rank.\\
Suppose $C$ is pure irreducible and $D_1 ,D_2$ two simple pure subcoalgebras. Then the intersection $D_1 \cap D_2$ is a non-trivial pure subcoalgebra. Hence we find that  $D_1 = D_2$. On the other hand, let $D$ be a unique pure simple subcoalgebra. According to (2) it is contained in every non trivial pure subcoalgebra of$C$. But this means that $C$ is pure irreducible.  
\end{proof}

\begin{definition} Let 
	  $C\in \mathcal{CA}^{flat}$.
	\begin{enumerate}
	 \item For $F\subset C$ pure simple, let $C(F)$ denote the pure irreducible component defined by $F$.
	
	\item Write $SP(C))$ for the set of simple pure subcoalgebras of $C$.
	\item Write $PC(C)$ for the set of pure irreducible components of $C$.
	\end{enumerate}

\end{definition}

\begin{remark}
The structure results from \cite{St}  above hold more generally over   Pr\"ufer domains or 	B\'{e}zout domains.
\end{remark}

\section{\large\textbf{Group like elements in flat coalgebras} }
\begin{definition}\label{irreducible group like}\begin{it}
		 Let
		   $C\in \mathcal{CA}$  with comultiplication $\Delta$ and counit $\epsilon$. An element $x\in C$ with\\
 $$\Delta (x)= x\otimes x$$  and 
 $$\epsilon (x) =1$$
 is called group-like.\\
 We denote  the set   of group-like elements in $C$  by $Gr(C)$. \end{it}\end{definition}

There is an adjunction between the category of sets $\mathbf{Set}$ and $\mathbf{C}^{flat}$.\[R(-):\mathbf{Set} \rightleftarrows \mathbf{C}^{flat} :Gr\]
where $R(-)$ is the functor which sends a set $S$ to the free $R$-module on $S$  with the coproduct induced by the diagonal and $Gr$ is the functor of $R$ points  $$Gr(C) = Hom_{ \mathcal{C}^{flat}} (R,C  )$$  or equivalently of group-like elements. If $R$ has no non trivial idempotents, the unit $$ 1\to R Gr (-)$$ is a natural isomorphism. This ajunction induces an adjuntion between $\mathbf{Set}$ and  $\mathbf{C}^{flat}_*$.\\

\begin{lemma}\label{g generates} Let
	 $D\in \mathcal{CA}^{flat}$ be $R$-free of dimension 1. Then $D$ is generated as a coalgebra by a unique group-like element $g\in Gr(C).$
	
\end{lemma}

\begin{proof} Let $d\in D$ be a generator of the $R$-module underlying $D$. Then $D\otimes D$ is generated by $d\otimes d$.
	Hence, $$\Delta (d)=\lambda d\otimes d$$ for some $\lambda\not= 0$. Now  $$d=(\epsilon \otimes I)(\Delta(d))=  \epsilon(\lambda d)d= \lambda d.$$ It follows that $\lambda =1 $.
\end{proof}

\begin{corollary} Let
	 $C\in  \mathcal{CA}^{^{flat}}$ and $D\subset C$ be a $R$-free  subcoalgebra of dimension 1. Then $D$ is a retract of $C$ via the counit $\epsilon $ .
\end{corollary}

\begin{proof} Let $g\in D$ be a group-like generator of $D$ and $\alpha:R \to D$ the map defined by $ 1\to g$.
	The the composition of coalgebra maps 
	$$R\stackrel{\alpha}{\to}  D\hookrightarrow  C \stackrel{\epsilon}{\to}R $$ is the identity on $R$.
\end{proof}




\begin{proposition}
	Let   $C\in \mathcal{CA}^{flat}$. There is a  1-1 correspondence between $Gr(C)$ and $R$-free 1-dimensional subcoalgebras of $C$. 
\end{proposition}
\begin{proof}
	Given any $g\in Gr(C)$ there is a unique $R$-free 1-dimensional subcoalgebra of $C$ generated by $g$. On the other hand, any $R$-free 1 -dimensional subcoalgebra is generated by a unique $g\in Gr(C)$ as we have seen in \ref{g generates}.
\end{proof}.

 \begin{lemma}
 	\label{inonGr}
 Let $C\in \mathcal{CA}^{flat}$. For each prime $p$ in $R$ there is an injection 
 $$(R/pR)\otimes R(Gr(C)) \to (R/pR)\otimes C .$$
 \end{lemma}

\begin{proof} By the first part of the proof of  \ref{group-like is direct} below  the group-like elements in $C$ and $(R/pR)\otimes C$  are linearly independent. It follows that the induced map $$Gr(C)\to Gr(C/pC)$$ is injective. Because $$(R/pR)Gr((R/pR)\otimes C)\to (R/pR)\otimes C$$ is an injection the same is true for the restriction on the direct factor $(R/pR)Gr(C)$.
	\end{proof}
 \begin{theorem}\label{group-like is direct}\begin{it} For  $C\in \mathcal{CA}^{flat}$ the group-like elements  in $C$ form a basis of a pure $R$-free subcoalgebra  $RGr(C)$ in $C$. \end{it}\end{theorem}
 
 \begin{proof} 
 	In the following, use is made of the fact that a finitely generated submodule of a flat module $M$ over  a principal ideal domain $R$ is free.
 
 	We show by induction that any finite subset of $Gr(C)$ is linearly independent. 	Suppose this is true for all subsets of cardinality less than or equal to $n$ and there is a relation 
 $$ r_{n+1}g_{n+1} = \sum^{n}_{ i=1}r_i g_i $$
 with all $r_i \not= 0$ and $g_j \in Gr(C)$ all distinct.
 Applying $\Delta $ to this equation one finds
 $$ (**)\hspace{2cm}r_{n+1}(\sum^{n}_{i,j =1}r_i r_j g_i \otimes g_j  )= \sum^{n}_{i=1}r_i g_i \otimes g_i .$$
Since the set  $\{ g_i | i\leq n \}$ is linearly independent in $C$ the subset $\{ g_i \otimes g_j |i,j \leq n \}$ is linearly independent in $C\otimes C$. Because $R$ is a domain, we find
that this equation forces $r_i r_j = 0$ for $i\not= j$ which implies  $n=1$. But then $r_1 =r_2 $  since $$1=\epsilon (g_{2} )= \epsilon (  g_1 ).$$ Which contradicts our assumption.
 	\\
 	This leaves us to show that the free submodule $Gr(C)$ is pure in $C$.
 By \ref{pure mod p} we have to check that the reduction mod $p$ of $$RGr(C) \to C$$ is injective for all primes $p$ in $R$.
 However, this follows from \ref{inonGr}.

 \end{proof}

\begin{remark} A weaker property than linear independence for the group like elements,
for coalgebras over a general ring, is proved by the authors  in \cite{D.G.M.}. 
\end{remark}

The basic result below was implicit in \cite{St}:

\begin{lemma}\label{basic}\begin{it} Let $C\in \mathbf{C}^{flat}$ and $D\subset C$ a subcoalgebra. The purification $\widetilde{D}$ is a coalgebra.\end{it}
\end{lemma}

\begin{proof}
	This follows from the isomorphism
	$$\widetilde{D\otimes D }\cong \widetilde{D}\otimes \widetilde{D }$$ proved in \cite[2.15]{St}.
\end{proof}



\begin{definition}\begin{it}
		Let $C\in  \mathbf{C}^{flat}$. The pure coradical $C_0 $ of $C$ is  defined to be the sum on all pure  simple subcoalgebras of $C$.
	\end{it}

	\end{definition}
	
	\begin{remark}In case $C\in \mathbf{C}_*^{flat}$ we have $C_0 = RGr(C)$.
	\end{remark}
We add the obvious: 

\begin{lemma}\begin{it}\label{important} Let $C\in\mathbf{C}_*^{flat} $ and $D\subset C$ a subcoalgebra. Then the  purification $\widetilde{D}$ is in $\mathbf{C}_*^{flat} $.\end{it}
\end{lemma}

\begin{proof} Let $S$ be a pure simple subcoalgebra of $\widetilde{D}$. Then $S$ is a pure simple subcoalgebra of $C$. Hence, isomorphic to $R$.
	
\end{proof}

\section{\large\textbf{Pure irreducible components} }

\begin{lemma}\label{irreducible direct}\begin{it}    The intersection of two different pure irreducible components of 
$C\in  \mathcal{CA}_*^{flat}$ 		is trivial.\end{it}\end{lemma}
\begin{proof}
	Let $P_{g_1}$ and $P_{g_2}$ two pure irreducible components with $g_1 , g_2$ the corresponding group-like elements. If  $P_{g_1}\cap P_{g_2}\not= 0$ then it is a pure  subcoalgebra. Hence it contains a pure simple  subcoalgebra $R_g$. It follows that $R_g =  R_{g_1} = R_{g_2}.$ This forces $g_1 = g_2$.
\end{proof}


\begin{lemma} \label{st}\begin{it} Let   $P_g\in  \mathcal{CA}_*^{^{flat}}$ be irreducible with group like element $g$ and $p\in R$ be prime. Then the reduction  mod $p$ map $$P_g \to P_g \otimes R/pR $$ induces a bijection
		$$\{g\} = Gr(P_g)=Gr(P_g\otimes R/pR)$$
	\end{it}
\end{lemma} 

\begin{proof} During the proof we  rely on the results and use the notation in \ref{oo}.
We do an induction on the coradical filtration $n$. We do an induction on the coradical filtration $C_n =C_n (P_g)$. For $n=0$ there is nothing to prove. We consider the case $n=1$ next.\\
By \ref{pr} there is an isomorphism of modules
$$C_1 \cong C_0 \oplus Pr(P_g)$$ where $Pr$ stands for the primitives see \ref{pr}.\\
Write any $x\in C_1$ as $rg + y$ with $y$ a primitive element. 
Then we have
$$\Delta (x)=r(g\otimes g) + g\otimes y + y\otimes g$$ 
Now write $\bar{x}\in P_g \otimes R/pR$ for the reduction of $x$. Suppose $\bar{x}$ is group like.\\ Then 
$$\Delta(\bar{x}) =\bar{x}\otimes \bar{x}$$
 and
$$\epsilon(\bar{x}) =1.$$ Hence
 $$\Delta (x)=u(x\otimes x )+ p(v\otimes w)$$
 where $v\otimes w$ stands for a general element in $P_g \otimes P_g$ and $\bar{u}=1$.
 We find that $$\bar{r}=1$$
 Now write
 $$x\otimes x = r^2 (g\otimes g) +rg\otimes y +r y\otimes g + y\otimes y$$ comparing coefficients in $R_g$ we get $r=1$ and that $u$ is a unit in the localization $R_{(p)} $ of $R$ at $p$.
 So we find
 $$\Delta (x)-u(x\otimes x  ) = p(v\otimes w)= -u(y\otimes y )$$
 It follows that the submodule generated by $y\otimes y \in C_1 (P_g )/ R_g\otimes C_1 (P_g )/R_g$ is divisible by $p$ in $P_g \otimes P_g$ and since all $U_j \subset U_{j+1}$ are pure, it is divisible in $U_2 /U_1 \cong C_1 \otimes C_0 \oplus C_1 \otimes C_1 \oplus C_0 \otimes C_1$. But then the one generated by
  $y$ in $Pr(C)$  is divisible by $p$.
   Hence there is $z\in C_1 (P_g)$ with $pz=y$. But this implies $\bar{x}=\bar{g}$ which is what we want.\\
 
 Suppose the assertion is true for the coradical filtration at indices less than or equal to $n-1$.\\
 Let again $x\in C_n$ and write $$x=rg +y$$  with $y\in C_n$ and $\epsilon (y)=0$.
 Suppose that $\bar{x}$ is group like.\\
 Write $$\Delta (x)= \sum_{i\leq n} c_i \otimes d_{n-i}$$
 with $c_j ,d_j \in C_j$.\\
 On the other hand, write\\
 $$x\otimes x = r^2 (g\otimes g) +rg\otimes y +r y\otimes g + y\otimes y$$ and
 
 $$\Delta (x) = u(x\otimes x )+p(v\otimes w) $$ with a unit $u$ as above.
 Again we find $r=1$ and hence
  $$\Delta (x) = u(x\otimes x )+p(v\otimes w)=u ( (g\otimes g) +g\otimes y + y\otimes g + y\otimes y)) +p(v\otimes w )$$
  
  It follows
  $$0=  u ( (g\otimes g) +g\otimes y + y\otimes g + y\otimes y)) +p(v\otimes w )-\sum_{i\leq n} c_i \otimes d_{n-i}$$
  
  So $$u(y\otimes y)-p(v\otimes w) \in U_n =\sum_{i\leq n} C_i \otimes C_{n-i}$$ 
  
  As a consequence,
  
  $$0=u(y\otimes y)-p(v\otimes w)  \in U_{2n}/U_{2n-1} $$
 Since $U_j$ is pure in $U_{j+1}$ for all $j$, we get that
 
  $y\otimes y$ is divisible by $p$ in $U_{2n}/U_{2n-1}. $ 
  Now use the direct sum description of $U_{2n}/U_{2n-1} $ to find that the class of $y\otimes y$ is in the direct factor $C_n /C_{n-1}\otimes C_n /C_{n-1}$ in $U_{2n}/U_{2n-1} $.\\
  
  It follows that $y$ is divisible by $p$ in $C_n / C_{n-1}$
  So  we can write
   $y=pz + c_{n-1} $ with $c_{n-1}\in C_{n-1}$.
  
The reduction of $x$ is thus given by the reduction of $g+c_{n-1}$.
By induction hypothesis, this implies that $\bar{x}=\bar{g}$.

\end{proof}
	
\begin{lemma} \label{st}\begin{it} Let   $C\in  \mathcal{CA}_*^{^{flat}}$. Let
  $g\in Gr(C)$ and $P_g $  the pure irreducible component of $C$ with $g\in P_g$. Then for each prime element $p\in R$ the reduction  mod $p$ of the inclusion  $$\oplus _{g\in Gr(C)} P_g\to C$$  is an injection.\end{it}
	\end{lemma}

\begin{proof}

	The reduction $P_g \otimes R/pR$ of the pure subcoalgebra $P_g$ lies is a irreducible component
	 $U_{\bar{g}} $ where $\bar{g}$ is the 
	 reduction of $g$.
 The reductions of  different  simple coalgebras corresponding to different $g\in Gr(C)$ remain different after reduction by \ref{group-like is direct}.\\
	 Moreover, there is an isomorphism
	$$\oplus_{{\bar{g}}\in Gr(C/pC))} U_{{\bar{g}}} \cong C\otimes R/pR $$  by the classical result for fields \cite[theorem 8.5.c]{S}.  The map in question is thus given as   
 $$\oplus_{g\in Gr(C)} P_g \otimes R/pR  \to \oplus_{\bar{g}\in Gr(C/pC))} U_{\bar{g}} \cong C\otimes R/pR .$$
  This proves the assertion.
	
\end{proof}

The next result is a direct consequence of \ref{pure mod p} and
\ref{st}
\begin{theorem}
\label{purecomp}\begin{it} For   $C\in \mathbf{C}_* ^{fla t}$  the inclusion of the direct sum of pure irreducible components in $C$
$$\oplus_g P_g \to C$$
is pure.\end{it}
\end{theorem}

 How the functor $Gr$ behaves if one passes to coalgebras over $K$ is answered in:
 
\begin{lemma}\label{Gr = GR_k}\begin{it}
		Let $C\in  \mathbf{C}_*^{flat}$. There is a bijection $$Gr(C)\cong Gr(C\otimes K)$$ induced by $i:C\to C\otimes K$.
\end{it}\end{lemma}

\begin{proof} Note first that a group-like element of $C$ gives rise to a unique group-like element in $C\otimes K$.\\
	Now let $g\in Gr(C\otimes K)$ be given. Then there is a non trivial $\lambda \in R$ with $$x:=\lambda g\in C.$$ We may assume that the subgroup $\langle x\rangle$ generated by $x$ is pure in $C$. Oterwise we pass to the purification   $\widetilde{\langle x\rangle}$. The counit $\epsilon$ is a non trivial homomorphism from $\langle x\rangle$ to $R$. The extension to $\widetilde{\langle x\rangle}$ is non trivial either. 
	
	  Hence, the rank 1  module $\widetilde{\langle x\rangle}$ is isomomorphic to $R$. 
	So we may replace $x$ by a generator.\\
	 Moreover, we may assume that $\lambda $ is not invertibel in $R$. Otherwise we are done.\\
	We have
	$$(i\otimes i)\Delta_C (x )=\Delta_{C\otimes K} (i(x))=\Delta_{C\otimes K} (\lambda g)=\lambda (g\otimes g)$$
	and
	$$(i\otimes i)(x\otimes x)=\lambda^2  (g\otimes g)$$
	
	Since $i\otimes i$ is injective we  find that
	
$$\lambda \Delta_C (x)=  x\otimes x$$
	
The submodule generated by $x\otimes x$ is pure in $C\otimes C$.

Hence	there is $\gamma \in R$ with 
	$$\Delta_C (x)=\gamma(x\otimes x) $$

	But then we have
	$$\gamma \lambda^2 (g\otimes g)=\lambda (g\otimes g) $$

	Now	apply $\epsilon \otimes 1$ and then $\epsilon$ to find
			 $$\gamma \lambda^2 =  \lambda$$
	
Since $R$ is a principal ideal domain,	it follows that $$\gamma =\lambda^{-1} .$$   Since $\lambda$ was   assumed non invertible in $R$  we arrive at a contradiction.
\end{proof}
	
	

\begin{lemma}\begin{it}\label{1234} Let $C=\sum_{i}C_i $ be a sum of pure  pointed coalgebras and $S$ a  simple pure subcoalgebra. Then $S\subset C_i$ for some $i$.\end{it}
	\end{lemma}

\begin{proof}
Let $S$ be a  pure simple subcoalgebra of $C$. Then $S$ is a retract of $C$ via $\epsilon $. We may assume that $S\subset C_1 + C_2$ since $S$ is of rank one.
Suppose $S\not \subset C_1 $ then $S\cap C_1 = 0$ since $S$ is pure simple. We must show that $S\subset C_2$.\\

Now let $g\in S$ be group-like.
Since by \ref{Gr = GR_k}  $$Gr(C)=Gr(C\otimes K)$$

we find $$g\in Gr(C_2 \otimes K)=Gr(C_2)$$ by \cite[8.0.3.]{S}.






\end{proof}

We need one more preliminary result.

\begin{proposition} \label{Kisok}\begin{it}
Let $C\in  \mathbf{C}_*^{flat}$ and write $\{P_g\}  $, $\{U_g\}$ for the pure irreducible components of $C$ and $C\otimes K$. There is an isomorphism for each $g\in G(C)=G(C\otimes K)$

$$(P_g)\otimes K \cong   U_g .$$
\end{it}
\end{proposition}

The proof needs some results on the coradical filtration given in section 8. below.

 \begin{proof}  The modules $P_g$ and $U_g$ are the filtered colimits of $C_i$ and $C_i \otimes K$ by \ref{4}. The assertion follows since $-\otimes K$ commutes with  filtered colimits.

\end{proof}
   
\begin{theorem}\label{irreducible generates}\begin{it} Any $C\in \mathbf{C}_*^{flat}$ is the direct sum of its irreducible components.\end{it}\end{theorem}

\begin{proof} In view of \ref{irreducible direct} and the fact that coalgebras of finite rank generate the category under filtered colimits, it is enough to show that a finite rank coalgebra in $ \mathbf{C}_*^{flat}$ is the sum of its irreducible subcoalgebras.\\
The pure simple subcoagebras of $C$ and  $C\otimes K$ are both in one one correspondence to the group-like elements of $C$ and  $C\otimes K$ which are in bijection to each other by \ref{Kisok}.\\
Let  $C_g$ be a pure irreducible component of $C$ with unique simple subcoalgebra $R_g$.
Then $C_g\otimes K$ has   simple subcoalgebra $K_g$. It is unique since it corresponds to the unique group-like element in  $g\in C_g\otimes K$.
Let $D$ be the sum of the pure irreducible components in $C$.  This sum is direct \ref{irreducible  direct} and finite. It is pure in   
 $C$ by  \ref{purecomp}. So $$D=\oplus_{g\in G(C)}D_g$$
 and $$D\otimes K = \oplus_{g\in G(C\otimes K)}(D\otimes K)_g$$
 The later is isomorphic to the sum of the irreducible components of $C\otimes K$ by \cite[theorem 8.0.5. c]{S}. If there were a non-trivial cokernel of the pure inclusion $$D\to C$$ it would be flat. But then the induced cokernel for $$D\otimes K\to C\otimes K$$ would be non-trivial as well. Since $D\otimes K$ is the direct sum of the irreducible components of $C\otimes K$ which equals $C\otimes K$, we are done.

\end{proof}

\begin{remark}
For $C\in \mathbf{C}_*^{flat}$ consider the partially ordered set of pure subcoalgebras $\pi(C)$. A path beween $S,T\in \pi(C)$ is a string of  morphisms $$S\to S_1 \leftarrow S_2 \to \ldots  \leftarrow T$$ with $S_i \in \pi (C)$. This defines a equivalence relation. The equivalence classes for this relation are called the path components of $C$. \ref{irreducible generates}  can be interpreted to say that the path components of $C$ are in one one correspondence with the irreducible components.
\end{remark}

\section{\large\textbf{Splitting off the pure coradical}}

\begin{lemma}\label{image is irreducible}\begin{it} Let $C\in \mathbf{C}_*^{flat}$ be pure irreducible with unique simple pure subcoalgebra $F$.
		Let $$f:C\to D$$ be a surjective morphism in $ \mathbf{C}_*^{flat} $. 
		\begin{enumerate}
			\item If $S$ is a non-zero pure simple subcoalgebra of $D$ then $f(R)=S $.
			\item $D$ is pure irreducible.
			
		\end{enumerate}\label{key}
		
\end{it}\end{lemma}
\begin{proof} 
	
	Pick a group-like  generator  $x \in F$ and let  $$f(x)=y.$$
	Note that $y$ is group-like.
	
	By \ref{Gr = GR_k} we know that 
	$$\{x\}=Gr(C)=Gr(C\otimes K) .$$
	
	Consider the surjective map
	$$f\otimes K:C\otimes K\to D\otimes K. $$
	
	and apply \cite[8.0.6.]{S} to find
	$$\{ y\otimes 1\}=Gr(D\otimes K).  $$
	By \ref{Gr = GR_k} there is a bijection
	$$Gr(D)\cong Gr(D\otimes K) $$ under which the element $y$
	corresponds to $y\otimes 1$. This proves (1).

	The assertion in (2) is now clear.  
	\\

	

\end{proof} 

	
	

\begin{lemma}\label{deco1}\begin{it}
	Let $f:C\to D$ be a morphism in $ \mathbf{C}_*^{flat}$  and $F\subset C$ a pure simple subcoalgebra. Then $f$ restricts to a morphism $$f|_{C(F)}\to D(f(F)))$$.\end{it}\end{lemma}

\begin{proof} Let $G=f(C(F))\subset D$. Then by \ref{image is irreducible}  $G$ is pure irreducible with pure simple subcoalgebra $f(F)$. Suppose $C(f(F))$ is the irreducible component of $D$ containing $f(F)$. We find $G\subset C(f(F))$. The assertion follows. 
	\end{proof}

\begin{lemma}\label{deco2}\begin{it} Any morphism $f:C\to D$  in $ \mathbf{C}_*^{flat}$ decomposes in the form:\end{it}
	
\xymatrix{C \ar[d]^{\cong}\ar[rr]^f{} &&
	D\ar[d]^{\cong} && 
	\\\oplus f_F \oplus_{F\in SP(C))} C(F) 
	\ar[rr]^{\oplus f_F} &&    \oplus_{K\in SP(D)} D(K) 
}	
	
	\end{lemma}

\begin{proof} This is a consequence of \ref{deco1} and  \ref{irreducible generates}
	\end{proof}
		
		
		
		Finally, we arrive at our main aim in this section.

\begin{theorem}\label{split}\begin{it}For $C\in\mathbf{C}_*^{flat} $ the inclusion $$\oplus_{F\in SP(C)} F \to C=\oplus_{F\in SP(C)}C(F)\cong C$$ splits naturally.\end{it}
\end{theorem}	
\begin{proof}	On each pure irreducible $C$ the splitting is given by the counit. Apply \ref{irreducible generates} to define the splitting for a general $C$. Naturality follows from \ref{deco2} above.
\end{proof}


The following examples show that one can not expect such a splitting for non pointed coalgebras.\\

\begin{example}Let $C=A^{\vee }$ be the $\mathbb{Z}$-linear dual of the finite  algebra $$\mathbb{Z}[ x ]/(x^2 - 2). $$
There is no algebra map from $A$ to $\mathbb{Z}$ since the root of $2$ is not an integer. There are no group-like elements in $C$ and we have to consider the map from the initial coalgebra
$$ R(\emptyset)=0\to C.$$
This  map is not split for a non trivial coalgera.

\end{example}

 One may try to work with deeply ramified rings without prime elements. For example the absolute integral closure of the integers  $\mathbb{Z}^+$ \cite{Hu} or local variants which figure in almost ring theory and p-adic Hodge theory \cite{Fa}.
 But this does not work either.
 \begin{example}
Let $C=A^{\vee }$ be the $\mathbb{Z}^+$-linear dual of the finite  algebra   $$A=\mathbb{Z}^+ [ x ]/(x^2 - 2). $$
In this case there are two algebra maps from $A$ to $\mathbb{Z}^+ $ defined by sending $x$ to 
$\sqrt{2}$ or $-\sqrt{2}$. The induced map
$$A\to  \mathbb{Z}^+ \times \mathbb{Z}^+ $$
is not surjective since the elements $(1,0)$ and $(0,1)$ are not in the image. The cokernel is $\sqrt{2}$-torsion. Hence the map induced on the duals is not split either. There is a short exact sequence:

$$\mathbb{Z}^+ \oplus \mathbb{Z}^+ \rightarrowtail C \twoheadrightarrow Ext^1 (\mathbb{Z}^+ /\sqrt{2}\mathbb{Z}^+ ,\mathbb{Z}^+ )\cong  \mathbb{Z}^+ /\sqrt{2}\mathbb{Z}^+$$ .

 \end{example}


\section{\large\textbf{Pure filtrations of coalgebras }}
\begin{definition}\begin{it} Let  $C\in \mathcal{CA}^{flat}$ and $D,F$ pure subcoalgebras of $C$. The wedge product of $D$ and $F$ is given by:
		$$D\wedge F := kernel(C \stackrel{\Delta}{\to} C\otimes C \to C/D\otimes C/ F )$$
\end{it}\end{definition}

\begin{definition}\label{filcoalg}\begin{it}
		A filtration of a coalgebra $C\in \mathcal{CA}^{flat}$ is a family of submodules
		$\{V_n \}^{\infty}_{n=0}=C$ 
		of $C$
		which satisfies 
		$$ V_0 \subset V_1 \subset V_2 \ldots V_{\infty}= \cup_{n=0}^{\infty} V_n$$ 
		
		and
		
		$$\Delta (V_n )\subset \sum_{i=0}^n V_{n-i}\otimes V_i  $$
		for all $n\geq 0$.\\
		
		We say that a filtration by submodules  $\{V_n \}^{\infty}_{n=0}$  is a $\wedge$-filtration if $$V_n \subset V_{n-1}\wedge V_0$$
		
		holds for all $n$.\\
	We say that the filtration is exhaustive if
$$C= V_{\infty}$$ holds.\end{it}
	
\end{definition}

It is not hard to see that the $V_i $ in a coalgebra filtration are in fact subcoalgebras. \\

		


The proof of the next lemma is direct from the definitions.

\begin{lemma}\begin{it}
		Let $$f:C\to D$$be a map in $\mathcal{CA}^{flat}$  and  $\{V_n \}^{\infty}_{n=0}$  a filtration of $C$.
		Then $\{  f(V_n ) \}^{\infty}_{n=0}$  is  a filtration of image($f)$.\end{it}
\end{lemma}

 \begin{lemma}
 	\begin{it}
	Let $C,D\in \mathcal{CA}^{flat}$   and $$f: C\to D$$ a surjective morphism. Let $\{C_j \}^{\infty}_{j=0}$ be a filtration of $C$ by pure subobjects.
	Moreover, assume 
	$$f(C_0 )=\widetilde{f(C_0 )}.$$
	Then $$D_0 \subset f(C_0 )$$ where $D_0 $ is the pure coradical.
	\end{it}
\end{lemma}

\begin{proof} Let $S\subset D$ be a pure simple subcoalgebra.
	Then there is $i$ such that $$ S\subset \widetilde{f(C_i )}.$$ In fact, by \ref{999} $$S\subset \widetilde{f(C_0 )}=f(C_0 ).$$

\end{proof}

\begin{lemma}\begin{it}\label{111} Let $U,V\subset D$ and $D\subset C$ be pure  subcoalgebras in $\mathcal{CA}^{flat}$.\begin{enumerate}
		\item 
	Then  $$U\wedge_D  V = (U\wedge_C V)\cap D.$$
	  \item Let $S\subset U\wedge_C  V$ be pure simple and non trivial.
	  Then $S\subset U$ or $S\subset V$.
	\end{enumerate}\end{it}
\end{lemma}

\begin{proof} Since $U, V$ are pure in $C$ and $D$, $D/U, D/V ,C/U, C/V$ are flat. Moreover,  $$D/U \to C/U$$ and $$D/V \to C/V$$ are pure injective. Consequently, $$D/U \otimes D/V \to C/U \otimes C/V$$
	is pure injective.\\ Now consider the commuting diagram in which the rows are short exact and the vertical morphisms are pure injective:
	
\xymatrix{U\wedge_D  V \ar[d]\ar@{>->}[rr]^{} &&
	D\otimes D \ar[d]\ar@{->>}[rr]^-{} && (D/U \otimes D/V)\ar[d] 
	\\
	U\wedge_C V  \ar@{>-}[rr] &&    C\otimes C 
	\ar@{->>}[rr]^-{} && (C/U \otimes C/V)
}

The assertion  (1) follows by a diagram chase.\\
By (1) 	$$S=S\cap (U\wedge_C V)= (S\cap  U)\wedge_S(V\cap S).$$
Since $S$ is non trivial, $S\cap U$ or $S\cap V$ is non trivial.
But because $S$ is pure simple $S=S\cap U$ or $S=S\cap V$.
	
\end{proof}

	



\begin{lemma}\label{999}
Let $C\in \mathcal{CA}^{flat}$  and $\{V_n \}^{\infty}_{n=0}$  an exhaustive $\wedge$-filtration by pure subcoalgebras. Then each pure simple subcoalgebra $S\subset C$ is contained in $V_0$.

\end{lemma}

\begin{proof} Since $S$  is of finite rank, it is contained in $V_n$ for some $n$. If  $n\geq 1$  then $$  S\subset V_n \subset V_0 \wedge V_{n-1} .$$  By \ref{111} , $S\subset V_0 $ or $S\subset  V_{n-1}$. It follows that $S\subset V_0 $  by induction.

\end{proof}

\begin{proposition}\label{987}\begin{it} Let $C\in \mathcal{CA}^{flat}$ and  $\{V_n \}^{\infty}_{n=0}$ be a $\wedge$ filtration of $C$. Then the purification 
	 $\{\tilde{V_n } \}^{\infty}_{n=0}$ is a $\wedge$-filtration of $C$.
	\end{it}
\end{proposition}

\begin{proof}
We must show that $$ \tilde{V_n }\subset \tilde{V}_{n-1} \wedge \tilde{V_0 }$$ for all $n$. Consider the commuting diagram:\\
\xymatrix{V_n \ar[d]\ar[rr]^{} &&
	C \ar[d]^{id}\ar[rr]^-{} && (C/V_{n-1} \otimes C/V_0  )\ar[d] 
	\\
	\tilde{V}_n   \ar[rr] &&    C 
	\ar[rr]^-{} && (C/\tilde{V}_{n-1} \otimes C/\tilde{V}_0 )
}

The composition of the upper row is trivial, since $\{V_n \}^{\infty}_{n=0}$ is a $\wedge$ filtration. For $x\in 	\tilde{V}_n $ there is $r\in R$ with $rx \in V_n $. Hence, the composition in the upper row is trivial on $rx$. But $(C/\tilde{V}_{n-1} \otimes C/\tilde{V}_0 )$ is flat, so the lower composition is trivial on $x$ as well. This shows the claim.
\end{proof}

\begin{lemma}\label{777}\begin{it}
	Let $$f:C \to D$$ be a surjective map in $\mathcal{CA}^{flat}$   and $C\in \mathcal{CA}^{flat}_*$. Then any pure simple subcoalgebra $S\subset D$ is a group-like coalgebra,
\end{it}\end{lemma}

\begin{proof}
Let  $\{U_n \}^{\infty}_{n=0}$  be the pure coradical filtration of $C$  defined in the section below and  $\{V_n \}^{\infty}_{n=0}$ the image filtration under $f$. We let  $\{W_n \}^{\infty}_{n=0}$ be the filtration od $D$ obtained by purification of   $\{V_n \}^{\infty}_{n=0}$. By \ref{987} and \ref{999}, every pure simple $S$ in $D$ is contained in $W_0$. Now $$f(C_0)=f(Gr(C))\subset Gr(D)$$.
Hence
$$\widetilde{f(C_0 )}=W_0 \subset \widetilde{Gr(D)}=Gr(C). $$

\end{proof}


	



\section{\large\textbf{The pure coradical filtration }}

\begin{definition}\begin{it}
		Let $C\in  \mathbf{C}^{flat}$. Let $C_0$ be the pure coradical of $C$.

		Define a  filtration $C_i$  of $C$, by submodules in an inductive way, by:
		
		$$C_n := C_{n-1}\wedge C_0$$
		This is called the pure coradical filtration.\\
\end{it}\end{definition}

\begin{lemma}\label{4}\begin{it} Let $C\in  \mathbf{C}_*^{flat}$ and $C_i $ be the terms in the pure coradical filtration.\\
		1) The inclusions $$C_i \to C$$
		are pure.\\
		2)	The inclusions  $$C_i \to C_{i+1}$$ are pure morphisms of coalgebras.\\
		3) These filtrations of $C$ and $C\otimes K$ are  related by an isomorphisms
		$$(C\otimes K)_i \cong (C_i )\otimes K$$ induced by the inclusion $C \to C\otimes K$.\\
		4) The filtration $C_i$ is exhaustive i.e.
		$$C=\cup C_i .$$\end{it}
\end{lemma}

\begin{proof}
	To see that 1) holds,	we do an induction on $i$. The case $i=0$ is settled in \ref{group-like is direct}. Suppose the assertion holds true for all $i<n$. Consider 
	$$C_n = kernel (C\to C/C_{n-1}\otimes C/C_0)$$
	Since $C_0$ and $C_{n-1}$ are pure, the tensor product on the right is flat. So the image of $C\to C/C_{n-1}\otimes C/C_0 $ is flat as well. This shows that $C_n $ is pure in $C$.\\
	

	Since the composition $$C_i \to C_{i+1} \to C$$ is pure by 1),
	the same holds for  $C_i \to C_{i+1}$. That the $C_i$ are a filtration of coalgebras and the inclusions are coalgebra morphisms is a consequence of \ref{CCK} and the fact that the statement holds over fields \cite[proposition 2.4.2 (e)]{Rad}. The assertion (2) is proved\\
	For 3) we do again an induction on $i$. The case $i=0$   follows from the equality $$Gr(C)=Gr(C\otimes K)$$ shown in \ref{Gr = GR_k}.

	Suppose the statement is true for all $i\leq n$.  Then  there is an isomorphism over $K$
	$$   (C/C_n \otimes C/C_0 )\otimes K \cong (C\otimes K) / (C\otimes K)_n \otimes ( C \otimes  K)/ (C\otimes K)_0$$
	Now, tensoring the  exact sequence of flat modules 
	$$C_{n+1}\to C \to C/C_n \otimes C/C_0 $$ with K, shows that there is an isomorphism 
	$$C_{n+1}\otimes K \cong (C\otimes K)_{n+1} $$ as needed.\\
	
	For coalgebras over fields, the corresponding statement to 4) is proved in Sweedler`s book. We reduce to it as follows.
	For $x\in C$ there is $i$ such that $$x\otimes 1 \in (C\otimes K)_i = C_i \otimes K $$ where the last equation holds  by 3). So there is $r\in R$ with $rx\in C_i$. But $C_i$ is pure in $C$ so we find $x\in C_i$.
\end{proof}

\begin{lemma}\label{CCK}
	\begin{it} Let $\{V_n \}^{\infty}_{n=0}$ be a filtration of submodules of  $C\in \mathcal{CA}^{flat}$ by pure objects with flat quotients $V_n /V_{n+1}$.  Assume  that  $\{V_n \otimes K \}^{\infty}_{n=0}$ is filtration of subcoalgebras of $C\otimes K$. Then  $\{V_n \}^{\infty}_{n=0}$ is already a filtration of subcoalgebras.
	\end{it}
\end{lemma}

\begin{proof} We have to show that $$\Delta (V_n) \subset \sum_{i=0}^n V_{n-i}\otimes V_i .$$
By \ref{oo} the filtration given by
$$ U_n  = \sum_{i=0}^n V_{n-i}\otimes V_i $$	is by pure submodules with $U_n$ pure in $U_{n+1}$. 
	
	 Consider the short exact sequence 
	$$ U_n \rightarrowtail  U_{n+1} \twoheadrightarrow ( U_{n+1} /U_n) $$ where the factor on the right is flat by purity.
	Next, contemplate the commuting diagram with short exact rows:\\
	
	\xymatrix{ U_n \ar[d]\ar@{>->}[rr]^{} &&
		U_{n+1} \ar[d]^{}\ar@{->>}[rr]^-{} && ( U_{n+1} /  U_n )\ar[d] 
		\\
		( U_n)\otimes K   \ar@{>->}[rr] &&    ( U_{n+1})\otimes K 
		\ar@{->>}[rr]^-{} && ( U_{n+1} / U_n )\otimes K)}

	If $\Delta (d)	\notin  U_n$ for some $d\in V_n$, then the projection of $\Delta (d)$ to $(( U_{n+1} / U_n ) )$ would be non trivial.
	But then it would be non trivial in $(( U_{n+1} / U_n ) )\otimes K$ which contradicts our assumption.

\end{proof}

\begin{definition} Let $C\in \mathcal{CA}_*^{flat}$ be
	 irreducible with unique $g\in Gr(C)$.
	 The submodule of primitive elements in $C$ is defined as:
	 $$Pr(C)=\{x\in C|\Delta(x)=g\otimes x+x\otimes g\}$$
\end{definition}

Here is another integral generalization of a classical fact:

\begin{lemma}\begin{it}\label{pr}Let $C\in \mathcal{CA}_*^{flat}$ be
	irreducible with unique $g\in Gr(C)$.
	Then $$C_1 = C_0 \oplus Pr(C)  $$ where $C_0 = R_g$ and $C_1$ are in the pure coradical filtration.\end{it}

\end{lemma}

\begin{proof} First, it is clear that $$Pr(C) \subset C_1.$$ We have $$ Pr(C)\subset kernel(\epsilon )$$ and $$\epsilon (g)=1.$$ so there is a direct sum decomposition of modules
	$$C_1 \cong C_0 \oplus P_1. $$
	We have seen that $$\Delta (C_1)\subset C_0 \otimes C_1 + C_1\otimes C_0 = R_g \otimes C_1 + C_1\otimes R_g .$$
	The remaining argument, which is an easy calculation, is word for word the one in \cite[proposition 10.0.1.]{S}.
	\end{proof}

\section{\large\textbf{Categorical properties of pointed coalgebras}}
 
 


\begin{lemma}\label{cocomplete}\begin{it} The category $\mathcal{CA}_*^{flat}$ 
		is cocomplete.\end{it}\end{lemma}
\begin{proof} Colimits in  flat coalgebras are defined by taking first the colimit in all $R$-modules followed by projection modulo torsion. It is enough to see that any colimit of pointed coalgebras in $\mathcal{CA}^{flat}$ is in $\mathcal{CA} ^{flat}_* $. This is clear for direct sums. For pushouts, let 
	\[
	\xymatrix{
		E \ar[r]\ar[d]_{} & F \ar[d]^{} \\
		D \ar[r]^{}   & F\oplus_{E} D}\]
	be a pushout diagram of pointed flat coalgebras.
	Consider the surjective morphism 
	$$p:F\oplus D \to  F\oplus_{E} D.$$
	Since $E\oplus D$ is pointed.
	the assertion follows from \ref{777}.
	

\end{proof}

The next result is needed, in order to see that $\mathcal{CA}^{flat}_*$ is closed under the tensor product.

\begin{proposition}\label{oo}\begin{it} Let $C, D\in \mathcal{CA}^{flat}$ and $\{C_n \}^{\infty}_{n=0}$ , $\{D_n \}^{\infty}_{n=0}$ be filtrations of $C,D$ by pure subobjects. \begin{itemize}
			\item 
		Then $\{U_n \}^{\infty}_{n=0}$ with $$ U_n = \sum_{i+j=n}C _i \otimes D_j $$ is  a pure filtration of $C\otimes D$.
		\item There is an isomorphism
	$$U_n / U_{n-1}\cong C_n / C_{n-1}\otimes D_0 \oplus C_{n-1}/ C_{n-2} \otimes D_1 /D_0 \oplus \ldots C_0 
	 \otimes D_n /D_{n-1} .$$
 \item If the filtrations $\{C_n \}^{\infty}_{n=0}$ , $\{D_n \}^{\infty}_{n=0}$ are exhaustive then so is $\{U_n \}^{\infty}_{n=0}$. \end{itemize}\end{it}
\end{proposition}

\begin{proof}

	Note that $U_n$ is the sum of $C_s \otimes D_t$ with $s+t \leq n$, and that 
	$$ C_s \otimes D_t \to C_i \otimes D_j $$
	are all  pure morphisms. Consider the pushout diagram
		\[
	\xymatrix{
		C_s \otimes D_t \ar[r]\ar[d]_{} & C_{s+1}\otimes D_t \ar[d]^{} \\
		C_s \otimes D_{t+1} \ar[r]^{}   & C_{s+1}\otimes D_t\oplus_{	C_s \otimes D_t} 	C_s \otimes D_{t+1} }\]
	By what we said above, the morphism 
	$$	C_s \otimes D_t \to  C_{s+1}\otimes D_t\oplus_{	C_s \otimes D_t} 	C_s \otimes D_{t+1}$$
	is pure. Moreover, by \ref{45} there is an isomorphism 
	$$	C_s \otimes D_t \cong (  C_{s+1}\otimes D_t )\cap (	C_s \otimes D_{t+1}) .$$ Hence, the pushout above is the sum of $C_{s+1}\otimes D_t$ and $C_{s}\otimes D_{t+1}$. Now consider:
	
		\[
	\xymatrix{
		C_n \otimes D_0 + C_{n-1}\otimes D_1 &+& C_{n-2} \otimes D_2  &\ldots&  C_1 \otimes D_{n-1}+C_0 \otimes D_n\\
		C_{n-1} \otimes D_{0} \ar[u]^{}&+& C_{n-2}\otimes D_1 \ar[u]\ar[ull]&\ldots &  C_0 \otimes D_{n-1}\ar[u]} 
	\]
	
	The quotient of the first arrow on the left is isomorphic to 
	$$(C_n /C_{n-1})\otimes D_0 \oplus C_{n-1}\otimes (D_1 / D_0) .$$
	
	Now $C_{n-2}\otimes D_1 /C_{n-2} \otimes D_0$ maps to $$(C_n /C_{n-1})\otimes D_0 \oplus C_{n-1}\otimes (D_1 / D_0) +  (C_{n-2} \otimes D_2 / C_{n-2} \otimes D_0 )\cong$$
	$$(C_n /C_{n-1})\otimes D_0 \oplus C_{n-1}\otimes (D_1 / D_0) + (C_{n-2} \otimes D_2 / D_0 )$$
	The resulting quotient is isomorphic to 
	$$(C_n /C_{n-1})\otimes D_0 \oplus C_{n-1}\otimes (D_1 / D_0) \oplus C_{n-2} \otimes (D_2 / D_1 ) .$$
	
	In an induction, we may continuing this argument as follows.
	Suppose that for some $s$ there is an isomorphism
	$$ U_{n,s} \cong \sum_{\stackrel{i+j=n}{j\leq s}}C _i \otimes D_j /\sum_{\stackrel{i+j=n-1}{j\leq s-1}}C _i \otimes D_j \cong   (C_n / C_{n-1})\otimes D_0 \oplus (C_{n-1}/ C_{n-2}) \otimes (D_1 /D_0 )\oplus \ldots C_{n-s } \otimes (D_s / D_{s-1}) $$
	We may assume that $s<n$. Consider  
	$$U_{n,s}+C_{n-s-1}\otimes D_{s+1} /(( C_{n-s-1}\otimes D_{s+1})\cap (\sum_{\stackrel{i+j=n-1}{j\leq s-1}}C _i \otimes D_j )) \cong $$ 
	$$U_{n,s}+C_{n-s-1}\otimes D_{s+1} /( C_{n-s-1}\otimes D_{s-1}) \cong  U_{n,s}+ C_{n-s-1}\otimes D_{s+1}/D_{s-1}$$ We have to take the quotient by $C_{n-s-1} \otimes D_{s} $. This quotient is isomorphic to 
	$$ (C_n / C_{n-1})\otimes D_0 \oplus( C_{n-1}/ C_{n-2}) \otimes(D_1 /D_0  ) \oplus \ldots (C_{n-s }/C_{n-1-s})  \otimes (D_s / D_{s-1}) \oplus  C_{n-s-1} \otimes (D_{s+1}/D_s )\cong U_{n,s+1}$$ which is what we want.
	 We arrive at the last  assertion.
	Since all the factors are  flat the assertion of the purity follows. \\
	The last assertion follows from the commutation of tensor products with filtered colimits.
	
\end{proof}

\begin{corollary}\label{789}\begin{it}
	Let $C, D\in \mathcal{CA}_*^{flat}$ and $\{C_n \}^{\infty}_{n=0}$ , $\{D_n \}^{\infty}_{n=0}$ and $\{(C\otimes D)_n  \}^{\infty}_{n=0}$  be the pure coradical filtrations.
	Then\\

	$$ (C\otimes D)_0 = C_0 \otimes D_0 $$.\\
	 Moreover, if $E \subset C\otimes D$ is a pure simple subcoalgebra. Then $$E\subset F\otimes G$$ where $F,G$ are pure simple subcoalgebras of $C$ and $D$ respectively.\end{it}
\end{corollary}

\begin{proof}
First,  $$ (C\otimes D)_0 \subset C_0 \otimes D_0 $$.follows from \ref{oo} and \ref{999}.
Note that $C_0 \otimes D_0 $ is the sum of $S\otimes T$ with $S,T$ pure simple in $C$ and $D$ respectively.  The other inclusion follows from the second assertion which is a consequence of \ref{1234}.
\end{proof}

The next corollary follows from \ref{789}.

\begin{corollary}\label{pointed}\begin{it}Suppose $C,D\in \mathcal{CA}_*^{flat}$. Then $C\otimes D$ is pointed.\\
	 Moreover, $$G(C\otimes D) = G(C)\otimes G(D)$$ holds true where $G(C)\otimes G(D):=\{ c\otimes d|c\in G(C), d\in G(D) \}$.\end{it}
\end{corollary}

\begin{lemma} \begin{it}The category $\mathcal{CA}^{flat}_*$ is locally presentable.\end{it}\end{lemma}

\begin{proof} It is proved in \cite{St}  that the category $\mathcal{CA}^{flat}$ is locally presentable. To be more precise, the generators are shown to be pure subcoalgebras of finite rank.
	Since pure subcoalgebras of pointed coalgebras are pointed this shows the claim
\end{proof}

\begin{corollary} \begin{it}The category $\mathcal{CA}^{flat}_*$ is complete.\end{it}\end{corollary}

The limits in $\mathcal{CA}^{flat}_*$ can be described more concretely.
The inclusion functor of  $$\mathcal{CA}^{flat}_* \to \mathcal{CA}^{flat}$$ has a right adjoint $ \Phi$ which sends $C\in \mathcal{CA}^{flat}$ to the sum of all pointed subcoalgebras. Now limits in  $\mathcal{CA}^{flat}_*$ are computed by composing the limit in  $\mathcal{CA}^{flat}$ with $ \Phi$.


Note that for  $C,D\in \mathcal{CA}$ the tensor product $C\otimes D$ is a coalgebra in a natural way. The tensor product restricts to a functor on $\mathcal{CA}^{flat}$ and
 $\mathcal{CA}^{flat}_*$. The later holds due to \ref{pointed}.
 Moreover, this operation defines the product in all these categories. \\	 

 




   
 
 \begin{theorem}\label{Barr}\begin{it} \begin{enumerate}\item The forgetful functor \[U:\mathcal{CA}_*^{flat} \to \mathbf{Mod}^{flat}\] has a right adjoint $S_* $.
 	 Moreover, $\mathcal{CA}^{flat}_*$ is comonadic over $\mathbf{Mod}^{flat}$.
\item The category $\mathcal{CA}_*^{flat}$ is cartesian closed.\end{enumerate}\end{it}
\end{theorem}
\begin{proof} The proof of (1) works as the one given by Barr  for $\mathcal{CA}$ in \cite{Barr}. The essential point is that colimits in $\mathcal{CA}_*^{flat}$ are created in $\mathbf{Mod}^{flat}$.
	
 To see (2), recall the description of the product as the tensor product and the fact that $C\otimes -$ respects all colimits in $\mathbf{Mod}^{flat}$ and hence in $\mathcal{CA}_*^{flat}$. Then   the special adjoint functor theorem provides the right adjoints for all functors $C\otimes -$.
	\end{proof}


\section{\large\textbf{Simplicial flat pointed coalgebras }}
Write  $s\mathbf{C}^{flat}_*$ for the category of simplicial objects in  $\mathbf{C}^{flat}_*$.\\
There is an induced adjunction between the category of simplicial sets $s\mathbf{Set}$ and $s\mathbf{C}^{flat}_*$.\[R(-):s\mathbf{Set} \rightleftarrows s\mathbf{C}^{flat}_* :Gr\]
 Since $s\mathbf{C}^{flat}_*$ is complete and cocomplete it is a simplicial category  \cite[II.2.5]{G.J.1}.\\
  For  $K\in s\mathbf{Set}$ and $C\in s\mathbf{C}^{flat}_*$, the tensor  is defined by \[K\otimes C = R(K)\otimes C.\]
 That  this makes sense i.e the functor $-\otimes C$ lands  in $s\mathbf{C}^{flat}_*$ follows from \ref{pointed}. The simplicial structure on $s\mathbf{C}^{flat}_*$ is then defined as in \cite{St}.\\
 
  Classes $\mathbf{we}$, $\mathbf{cof}$ of morphisms in $s\mathbf{C}^{flat}_* $  are defined as homology equivalences and maps which are degree wise pure injections respectively.
 The class $\mathbf{fib}$ of fibrations is defined by the right lifting property with respect to trivial cofibrations.\\
 
 The main result of this section is:\\

\begin{theorem}\label{main}\begin{it}  The classes $\mathbf{we}$,  $\mathbf{cof}$ and  $\mathbf{fib}$ define a left proper, combinatorical, simplicial,
		monoidal  model category  structure on $s\mathbf{C}^{flat}_* $.\end{it}
\end{theorem}

The proof is  mainly the same proof as in the unpointed case given in \cite{St}. We indicate some modifications  in the argument.

 We have seen that  $s\mathbf{C}^{flat}_* $ is complete and cocomplete,
	For the factorization of a map into cofibration and trivial fibration one just replaces $S$ by $S_* $ and the adjunction to the model structure of flat simplicial modules $s\mathbf{Mod}^{flat}$ constructed  in \cite{St} in the argument.


We turn to the other factorization axiom.

Let $\lambda$ be a regular infinite cardinal which is least as big as the cardinality of $R$.

\begin{lemma}
	\label{pure}\begin{it}
		Let  $D\in s\mathbf{C}^{flat}_*$. Every element $x\in D$ is contained in a pointed simplicial subcoalgebra $C$ which is pure in $D$ and  of cardinality less than or equal to $\lambda$.  
\end{it}\end{lemma}

\begin{proof}
Suppose $x$ is in simplicial degree $n$.	Choose a pure   subcoalgebra of finite rank $D_x \subset D_n$ with $x\in D_{x,n}$. It is of cardinality less than or equal to $\lambda$. 
 This is possible by \ref{colimits}.
  The purification $\widetilde{D_x}$ of the simplicial subcoalgebra generated by $D_{x,n}$.  Note that since the degenerate simplices form a direct summand, the purification acts non trivially only in simplicial degrees less than  $n$. Then $\widetilde{D_x}$   is a pure  simplicial subcoalgebra using \ref{basic}.  It has still cardinality less than or equal to $\lambda$ and of finite rank. Moreover, it is pointed  by \ref{important}.
  
\end{proof}

From this, we verify the bounded cofibration axiom  of Goerss and Jardine \cite{G.J.2}.

\begin{lemma}\label{key}\begin{it} Let  $f:A \to B$ be a trivial cofibration in  $s\mathbf{C}^{flat}_*$ and let $x\in B$ be any element.
		Then there is a simplicial  subcoalgebra $B_x$ of $B$ with $x \in B_x$ which is of cardinality $\leq \lambda$   such that $A\cap B_x$ is a subcoalgebra of $A$
		and such that \[f_| :A\cap B_x \to B_x \] is a trivial cofibration.\end{it}
\end{lemma}

\begin{proof} We construct inductively a chain of pure subcoalgebras $B^n \subset B$ with 
	$$\pi_* (B^i /(B^i \cap A) )\to  \pi_* (Bi+1^{} /(B^{i+1} \cap A) )$$ trivial,\\
	Choose a simplicial subcoalgebra $B^0$  of $B$ which is of cardinality $\leq \lambda$ which is pure in $B$  
	and with $x\in B^0$. This is possible by \ref{pure}.\\
	Suppose $B^n$ has been found. Let $[z]\in \pi_k (B^n /(B^n \cap A) )$ there is then $y\in B/A$ with $\partial (y)= z$ in $B/A$. This is possible since $\pi_* (B /A )=0$. Let $\bar{y}$ be a representative of $y$ in $B$. 
	Choose a pure simplicial subcoalgebra $N_z \subset B$ of cardinality $\leq \lambda$  with $\bar{y}\in N_z$.
	Then $[z]$ is in the kernel of \[\pi_k (B^n /(B^n \cap A) )\to \pi_k (\widetilde{B^n+N_z } /(\widetilde{(B^n +N_z )} \cap A)) ).\]
	Set \[B^{n+1}=\widetilde{B^n + \sum_{z\in \pi_{*} (B^n /B^n
	\cap A ) } }N_z.\]
	By \ref{intersections},\ref{basic}, 
	 and \ref{important},  $B^{n+1 }$ is a pointed pure simplicial coalgebra. The same is true for $B^{n+1}\cap A$ .\\
	Finally, let $B_x = \cup_n B^n$.
	By construction $\pi_* (B_x /(B_x \cap A))$ is trivial. Moreover, the union of a chain of pure submodules is a pure submodule.
	Since $A$ and $B_x$ are   pointed simplicial  subcoalgebras so is $B_x \cap A$. We conclude that  \[f_| :A\cap B_x \to B_x\] is a trivial cofibration with all properties asserted.    
\end{proof}

Using \ref{key} the same proof as \cite[3.7.]{G} shows:

\begin{lemma}\label{Bleq}\begin{it} A map \[f:C\to D\] in $s\mathbf{C}^{flat}_*$ is a fibration if and only if it has the right lifting property with respect to trivial cofibrations
		$A\to B$ with $B$ of cardinality less than or equal to $\lambda$.\end{it}\end{lemma}

From \ref{Bleq} we get by the small object argument:

\begin{lemma}\label{fac2}\begin{it} Any map \[f:C\to D\] in $s\mathbf{C}^{flat}_*$ can be factored $f=pi$ into a trivial cofibration $i$  and a fibration $p$.\end{it}
\end{lemma}

The proof that the retract axiom and the lifting axioms holds is as the one in the unpointed case given in \cite{St}.
The same is true for the additional properties of   the model category $s\mathbf{C}^{flat}_*$ listed in \ref{main}.

\begin{remark}In \cite{GeGo} Getzler and Goerss show that coassociative space like differential graded coalgebras over a field form a cofibrantly generated model category.  A space-like coalgebra is a differential graded coalgebra which is group like in degree zero.  
\end{remark}

\begin{remark} In \cite{L} Lurie defines an $\infty$-topos to be of homotopy dimension zero if every object admits a map from the terminal object. This makes sense for the $\infty$-category defined by $s\mathbf{C}^{flat}_*$.  Since every $C\in s\mathbf{C}^{flat}_*$ admits a map from the final simplicial coalgebra $R$, $s\mathbf{C}^{flat}_*$ is of homotopy dimension $0$. This is not true for  $s\mathbf{C}^{flat}$.  
\end{remark}	

 \section{\large\textbf{Proof of the  main theorem}} 
 
We may assume,  in this section, that the ground ring $R$ is a prime field or the localization of the integers $\mathbb{Z}[S^{-1}] $ at a set of primes $S$. This implies no loss of generality 
by the results in \cite[chapter I, section 9]{B.K}. 

We let $\bold{Ho}_R (s\mathbf{Set}) $ and $\bold{Ho}( s\mathbf{C}^{flat}_*)$ denote the homotopy categories of the model categories of the Bousfield localization of simplicial sets at $H_* (-;R)$ equivalences and the one of \ref{main} respectively.\\

Using the fact that every object in  $s\mathbf{C}^{flat}_*$ is cofibrant the proof of the next result is idenitcal to the one of \cite[proposition 4.4.]{G} Note because of this fact,  one has $\bold{L}R(-) = R(-). $

\begin{proposition}\begin{it}  The functors $R(-)$ and $Gr$ induce adjoint functors $$\bold{L}R(-):\bold{Ho}_R (s\mathbf{Set})\rightleftarrows  \bold{Ho}( s\mathbf{C}^{flat}_*):\bold{R}Gr$$\end{it}
	
\end{proposition}

The key result  on which everything  depends is:\\

\begin{proposition}\begin{it}\label{every } The functor $Gr$ respects all weak equivalences.\end{it}
	\end{proposition}

\begin{proof} Let $f:C \to D$ be a weak equivalence in $s\mathbf{C}^{flat}_*$. Then using \ref{split} the map $R(G (f))$ is a retract of $f$.
	Hence, it is a weak equivalence in $s\mathbf{C}^{flat}_*$. But this means that $G (f)$ is a weak equivalence.
\end{proof}

We are now in the position to prove \ref{mainmain}:

\begin{proof} of \ref{mainmain}
It is enough to show that the derived unit $$X\to \bold{R}Gr(R(X) $$ is a natural isomorphism in $\bold{H}_R (s\mathbf{Set})$.
Let $$j: R(X) \to R(X)_f $$ be a natural fibrant replacement.
Then $Gr (j)$ is a weak equivalence. Since $G_R$ respects fibrations and weak equivalences we have a weak equivalence with fibrant target
$$X\cong Gr (R(X)) \to Gr (R(X)_f .$$
This induces a natural isomorphism on $\bold{H}_R (s\mathbf{Set})$. 

\end{proof}

The proof also shows: 

\begin{corollary}\begin{it} The unit of the adjunction 
	$$X \to \bold{R}Gr (R(X) )$$
	is the Bousfield $H_* (-;R)$ localization of $X$.\end{it}
	\end{corollary}


 \section{\large\textbf{Pointed coalgebras, binomial algebras and Galois action}} 
 
\begin{definition}
A torsion free ring $A$ is called binomial if
it satisfies  for each rational prime $p$ :
\begin{enumerate}	
	  .\item $A/pA$ is reduced
	\item  Every residue field of $A/pA$ is isomorphic to $\mathbb{F}_p$.  
\end{enumerate}
\end{definition} 
This is not the original definition of a binomial ring. That the given one  is equivalent to the original one is proved in  \cite[theorem 3.4. ]{Elliot}.

It is then clear that the dual of  a integral flat pointed coalgebra  $C$ satisfies (2) but not (1) in general. 
The dual of the mod $p$ reduction of a generator of finite rank is an artinian algebra. Hence it is a finite product of finite local rings with nilpotent maximal ideal. The presence of nilpotent elements adds flexibility since it makes classical deformation arguments possible. This argument  shows how much bigger   the category of flat pointed coalgebras is  than its subcategory of
simple coalgeras.\\

Let $\mathbb{F}$ be a non algebraiclly  closed field with algebraic closure $\mathbb{K}$ and absolute Galois group $G$. Denote by $Set_G$ the category of discrete $G$-sets.
It is a theorem of Goerss that the pair of adjoint functors 
 $$\mathbb{K}(-)^G : s\mathbf{Set}_G\rightleftarrows   s\mathbf{C}_{\mathbb{K}}:Gr_G$$

induce, for appropriate model category structures, a full and faithful functor of the homotopy categories of
 simplicial $G$-sets to the homotopy category of simplicial coalgebras. Here $$Gr_G (C):= Hom_{\mathbf{C}_{\mathbb{K}}} (\mathbb{K}, C\otimes K).$$

The question which $G$-sets correspond to a pointed coalgebra under $Gr_G$ is answered in:

\begin{proposition}
	Let $C$  be a pointed $\mathbb{F}$-colagebra.
	Then $G$ acts trivially on $Gr_G (C)$.
\end{proposition}

\begin{proof}
We may assume that $C$ is of finite rank. Then the dual $C^* =A$  is isomorphic to a finite product of finite local algebras with nilpotent maximal ideals and residue fiels $\mathbb{K}$.
$$A\cong \prod_{i\leq n} A_i$$ Moreover, we may transfer the argument to algebras of of finite rank $\mathbf{A}_{\mathbb{K}}$.
\\
Since $\mathbb{K}$ has no zero divisors there is a bijection
$$ Hom_{\mathbf{A}_{\mathbb{K}}}(A,\mathbb{K})\cong \coprod_{i\leq n}Hom_{\mathbf{A}_{\mathbb{K}}}(A_i,\mathbb{K})  $$
Because $\mathbb{K}$ has no nilpotents every algebra morphism
$$f:A_i \to \mathbb{K}$$
factors uniquely over $\mathbb{K}$. Hence there is a decomposition of $G$-sets
$$Hom_{\mathbf{A}_{\mathbb{K}}}(A,\mathbb{K})\cong\coprod_{i\leq n}Hom_{\mathbf{A}_{\mathbb{K}}}(\mathbb{K},\mathbb{K})  $$
This is a trivial $G$-set.
\end{proof}








\begin{thebibliography}{1}
\bibitem{A.R.}
Ad\'{a}mek J.; Rosik\'{y} J. Locally presentable and accessible categories.
London Math. Soc.Lect. Note Series, No 189, Cambridge University Press, (1994).
\bibitem{BaBu}
Bachmann  T;  Burklund R. E-infinity-coalgebras and p-adic Homotopy Theory. arXiv:2402.15850 .
\bibitem{Barr}
Barr M.
Coalgebras over a commutative ring. J. Algebra 32 (1974), no. 3, 600--610.
\bibitem{B.R.}
 Borceux F.;  Rosik\'{y} J. Purity in algebra. Algebra universalis,  56, Issue 1, (2007), 15--35.
 \bibitem{Bou}
 Bousfield A.K. The localization of spaces with respect to homology.
 Topology 14, (1975), 133--150.
\bibitem{B.K}Bousfield, A. K.; Kan, D. M.  Homotopy limits, completions and localizations, Lecture Notes in Mathematics, vol. 304, Springer-Verlag (1972).
\bibitem{B.W.}
Brzezinski T.; Wisbauer R. Corings and comodules. London Mathematical Society Lecture Note Series, 309. Cambridge University Press, Cambridge, (2003). xii+476 pp.
 \bibitem{Do}
 Dold A. Lectures on algebraic topoloy.Grundlehere der mathematischen Wissenschaften 200, Springer Verlag (1980).
 
 \bibitem{D.G.M.}
  Duchamp  G.,  Grinberg D.,  Minh V. Three variations on the linear independence of grouplikes in a coalgebra,
 arXiv:2009.10970 [math.QA]  (2020)
 
\bibitem{E.J.}Enochs E.E.; Jenda, O. M. G. Relative homological algebra. De Gruyter Expositions in Mathematics, 30. Walter de Gruyter, Berlin, (2000). xii+339 pp.
\bibitem{Fa}
Faltings G.; p-Adic Hodge Theory. Journal of the Amer. Math. Soc. Vol. 1, No. 1, (1988) 255--299.
\bibitem{Elliot}
 Elliott J. Binomial rings, integer-valued polynomials, and $\lambda $-rings. Journal of pure and applied Algebra,
207(1):165–185, 2006.
\bibitem{F.S.1}
Fuchs L.; Salce L. Modules over non-Noetherian domains. Mathematical Surveys and Monographs, 84. American Mathematical Society, Providence, RI, (2001). xvi+613 pp.
 \bibitem{GeGo}
 Getzler E.; Goerss P.
 A model structure for differential graded coalgebras. Preprint (1999) available on the homepage of  Paul Goerss.
\bibitem{Gil 1}
Gillespie J. The flat model structure on Ch(R). Trans. Amer. Math. Soc. 356 (2004), no. 8, 3369--3390. 
  \bibitem{G}
   Goerss P.G. Simplicial chains over a field and p-local homotopy theory. Math. Z. 220 (1995), no. 4, 523--544.
  \bibitem{G.J.1}
   Goerss, P. G.; Jardine J. F. Simplicial homotopy theory. Progress in Mathematics, 174. Birkh\"auser Verlag, Basel, (1999). xvi+510 pp.
   \bibitem{G.J.2}
    Goerss, P. G.; Jardine, J. F. Localization theories for simplicial presheaves. Canad. J. Math. 50 (1998), no. 5, 1048--1089. 
   \bibitem{Hi}
   Hirschhorn P. S. Model categories and their localizations. Mathematical Surveys and Monographs, 99. American Mathematical Society, Providence, RI, (2003). xvi+457 pp.
\bibitem{Horel}
 Horel G. Binomial rings and homotopy theory, Journal f"ur die reine und angewandte Mathematik
(Crelles Journal), vol. 2024, no. 813, 2024, pp. 283-305,
\bibitem{Hovey2}
Hovey M. Model categories. Mathematical Surveys and Monographs, 63, Amer. Math. Soc., Providence, RI, (1999).
\bibitem{Hu}
Huneke C. Absolute Integral Closure. Contemporary Math.vol. 555 (2011), 119--135.
\bibitem{Kr}
Kriz I.  p-adic homotopy theory, Topology and its Applications 52 (1993), 279–308.
\bibitem{La}
Lam, T.Y.  Lectures on Modules and Rings. Graduate Texts in Mathematics. Vol. 189. New York, NY: Springer New York (1999).
 \bibitem{L}
 Lurie J. Higher Algebra. Available at Homepage.
 \bibitem{Ma2}
  Mandell M.
 E-infinity Algebras and p-Adic Homotopy Theory.
 Topology 40 (2001), no. 1, 43-94.
\bibitem{Ma}
Mandell, M.
Cochains and homotopy type. 
Publ. Math. Inst. Hautes \'{E}tudes Sci. No. 103 (2006), 213--246.
\bibitem{Ma3}
Algebraic Models for Homotopy Types IV
Algebraic Models for Integral Homotopy Types
Young topologists meeting 2013.
 \bibitem{Per}
 Percey A. Some relations between non-stable integral cohomology operations, 
 Bull. Korean Math. Soc. 47 (2010) 275–-286
  \bibitem{P}
 Prest M. Purity, spectra and localisation. Encyclopedia of Mathematics and its Applications, 121. Cambridge University Press, Cambridge, (2009). xxviii+769 pp. 
\bibitem{Q}
D.~G. Quillen. 
\newblock {\em Homotopical algebra.}
\newblock {Lecture Notes in Mathematics},
  No. 43, Springer-Verlag, Berlin, (1967).
  \bibitem{Q2}
   Quillen D. Rational homotopy theory, The Annals of Mathematics, Second Series, Vol. 90, No. 2 (Sep., 1969), pp. 205-295.
 \bibitem{Rad}   Radford D.E.. Hopf Algebras, World Scientific, Series on Knots and Everything, Vol. 49.  
  \bibitem{R}
Raptis G. Simplicial presheaves of coalgebras.
Algebraic and Geometric Topology 13 (2013). 1967-2000.
\bibitem{R.R.}
Raptis G.and M. Rivera M. The simplicial coalgebra of chains under three different notions of weak
equivalence, IMRN, no. 16, 11766-11811 (2024).
\bibitem{R.W.Z.}
 Rivera M.,  Wierstra F., and  Zeinalian M. The simplicial coalgebra of chains determines
homotopy types rationally and one prime at a time. Trans. Amer. Math. Soc. 375 (2022), no.
5, 3267–3303.
\bibitem{St}
 Stelzer M., Purity and homotopy theory of coalgebras,
Journal of Pure and Applied Algebra 223 (2019) 2455–2473.
  \bibitem{S}
   Sweedler M. E. Hopf algebras. Mathematics Lecture Note Series W. A. Benjamin, Inc., New York (1969) vii+336 pp. 
   \bibitem{Su}
   Sullivan, D. Infinitesimal computations in topology.
   Publications Mathématiques de l'IHÉS, Volume 47 (1977), pp. 269-331.
  \bibitem{Y}
  Yuan A. Integral Models for Spaces via the Higher Frobenius. Journal of the American Mathematical Society (2023).
  \end{thebibliography}

\end{document}